\documentclass[twoside]{amsart}
\usepackage{latexsym}
\usepackage{amssymb,amsmath,amsopn}
\usepackage[dvips]{graphicx}   
\usepackage{color,epsfig}      
\usepackage{tikz} 
\usepackage{subcaption} 
\usepackage{eurosym}
\usepackage[utf8]{inputenc}
\usepackage[T1]{fontenc}

\newtheorem{definition}{Definition}

\newtheorem{theorem}{Theorem}

\newtheorem{remark}{Remark}

\newtheorem{lemma}{Lemma}
\newtheorem{problem}{Problem}
                    
\newcommand{\F}{\mathcal{F}}

\DeclareMathOperator{\diam}{diam}
\DeclareMathOperator{\inter}{int}
\DeclareMathOperator{\cl}{cl}
\DeclareMathOperator{\conv}{conv}
\DeclareMathOperator{\area}{area}
\DeclareMathOperator{\perim}{perim}

\newcommand{\B}{\mathbf{B}}
\renewcommand{\Re}{\mathbb{R}}

\begin{document}

\title[Monohedral tilings]{On monohedral tilings of a regular polygon}
\author[B. Basit]{Bushra Basit}
\address{Department of Geometry, Institute of Mathematics,\\
Budapest University of Technology and Economics,\\
M\H uegyetem rkp. 3., H-1111 Budapest, Hungary} 
\email{bushrabasit18@gmail.com}

\author[Z. L\'angi]{Zsolt L\'angi}
\address{Department of Geometry, Institute of Mathematics,\\
Budapest University of Technology and Economics,\\
M\H uegyetem rkp. 3., H-1111 Budapest, Hungary}
\address{MTA-BME Morphodynamics Research Group,\\
M\H uegyetem rkp. 3., H-1111 Budapest, Hungary} 
\email{zlangi@math.bme.hu}

\thanks{Partially supported by
the BME Water Sciences \& Disaster Prevention TKP2020 Institution Excellence Subprogram, grant no. TKP2020 BME-IKA-VIZ,
the NKFIH grant K134199, the J\'anos Bolyai Research Scholarship of the Hungarian Academy of Sciences, and the \'UNKP-20-5 New National Excellence Program by the Ministry of Innovation and Technology.}


\begin{abstract}
A tiling of a topological disc by topological discs is called monohedral if all tiles are congruent. Maltby (J. Combin. Theory Ser. A 66: 40-52, 1994) characterized the monohedral tilings of a square by three topological discs. Kurusa, L\'angi and V\'\i gh (Mediterr. J. Math. 17: article number 156, 2020) characterized the monohedral tilings of a circular disc by three topological discs. The aim of this note is to connect these two results by characterizing the monohedral tilings of any regular $n$-gon with at most three tiles for any $n \geq 5$.
\end{abstract}

\maketitle

\section{Introduction}

Subsets of the Euclidean plane $\Re^2$ homeomorphic to the Euclidean closed circular unit disc $\B^2$ centered at the origin $o$  are usually called \emph{topological discs} or \emph{Jordan regions}.\footnote{The set $\B^2$ is the set of points in the plane whose Euclidean distance from $o$ is at most one. To distinguish them from topological dics, we call the sets similar to $\B^2$ circular discs, or Euclidean circular discs.} A family of topological discs $\{ D_1, D_2, \ldots, D_k\}$ whose union is a topological disc $D$ and whose elements are mutually nonoverlapping (i.e. their interiors are mutually disjoint), is called a \emph{tiling}, \emph{decomposition}, or \emph{dissection} of $D$, and the elements of the family are called \emph{tiles}. A tiling is called \emph{monohedral}, if all tiles are congruent to a given topological disc, which is often called \emph{prototile} \cite{FGL21}.

The history of the investigation of tilings goes back to ancient times and well beyond the boundary of mathematics (see e.g. \cite{GS90, Schulte}). The aim of this paper is to examine one such problem. A result of Maltby \cite{Maltby} in 1994 states that a square cannot be dissected into three non-rectangular congruent topological discs. Along the same line, Yuan, C. Zamfirescu and T. Zamfirescu \cite{Zamfirescu} proved, answering a question of Danzer, that in any monohedral tiling of a square by five \emph{convex} tiles, the prototile is a rectangle, and conjectured that the same holds if the number of tiles is an odd prime. This question has been recently answered in \cite{RRW2020} in the special case that the prototile is a $q$-gon with $q \geq 6$ or it is a right-angled trapezoid, and a computer-assisted proof has been given in \cite{MR2021} for seven or nine tiles.

We intend to investigate a similar question, also based on the result of Maltby in \cite{Maltby}. To state our main result, we call a monohedral tiling of a regular $n$-gon $P$, centered at the origin $o$, by tiles $D_1, D_2, \ldots, D_k$ \emph{rotationally generated} if the rotation around $o$ and with angle $\frac{2\pi}{k}$ leaves $P$ invariant, and permutes the tiles (cf. Figure~\ref{fig:rotationally_generated}). 

\begin{theorem}\label{thm:main}
Let $P$ be a regular $n$-gon with $n \geq 5$, and let $\F$ be a monohedral tiling of $P$ by $k$ topological discs, where $2 \leq k \leq 3$. Then either $k=2$, $n$ is odd and $\F$ contains the two halves of $P$ dissected by a line of symmetry of $P$, or $n$ is divisible by $k$ and $\F$ is rotationally generated. 
\end{theorem}

\begin{figure}[ht]
\begin{center}
 \includegraphics[width=0.35\textwidth]{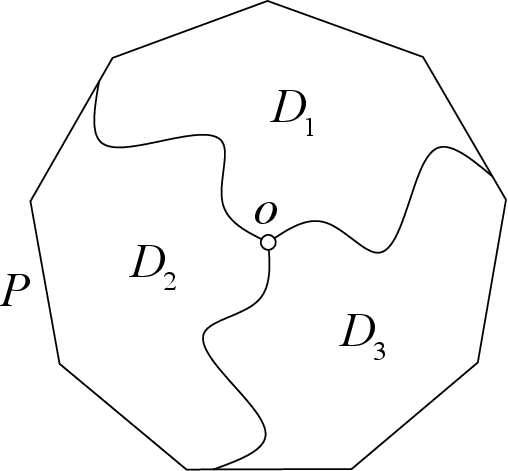}
 \caption{A rotationally generated tiling of a regular $9$-gon $P$ with three tiles}
\label{fig:rotationally_generated}
\end{center}
\end{figure}

We note that the same theorem with the Euclidean circular disc in place of $P$ was proved in \cite{KLV2020}, and monohedral tilings, with at most $3$ tiles, of a convex disc with strictly convex and smooth boundary were partially characterized in \cite{NV}. Theorem~\ref{thm:main} can be regarded as a result connecting the one in \cite{Maltby} for squares and the one in \cite{KLV2020} for circular discs. The proof of Theorem~\ref{thm:main} is based on (geometric, combinatorial and topological) tools from both \cite{KLV2020} and \cite{Maltby}, and also on some new ideas.

Finally, we remark that in the past few years a `dual version' of this problem, namely the investigation of dissecting the Euclidean plane into mutually \emph{incongruent} tiles with equal area under various constraints has also gained significant interest. For related results the interested reader is referred to \cite{Frettloh, FGL21, F-R, F-R2, KPT18_1, KPT18_2}.
The number of dissections of a square into equal area rectangles was estimated in \cite{BF88, HLL83}.
For the investigation of monohedral dissections of geometric figures using a different notion of dissection, see e.g. \cite{Laczkovich, Kiss}.

The structure of our paper is as follows. In Section~\ref{sec:prelim} we introduce the necessary notation and tools to prove our main result. In Section~\ref{sec:proof} we present the proof of Theorem~\ref{thm:main}. Finally, in Section~\ref{sec:remarks} we collect some additional remarks.

\section{Preliminaries}\label{sec:prelim}

In the paper, for any set $X \subset \Re^2$, we denote the interior, the boundary, the closure and the convex hull of $X$ by $\inter(X), \partial X$, $\cl(X)$ and $\conv (X)$, respectively.
Furthermore, if $X$ is bounded and nonempty, then $\diam(X)$ denotes its diameter. For any $x,y \in \Re^2$, by $[x,y]$ we denote the closed segment with endpoints $x,y$. By a simple curve we mean a continuous curve which does not cross itself, and a simple, closed curve is a simple curve whose two endpoints coincide. With a little abuse of notation, we call the points of a simple, not closed curve, different from its endpoints, interior points of the curve.
Finally, for brevity, we call a topological disc simply a \emph{disc}.

In the proof, $n \geq 5$ and $P$ denotes a regular $n$-gon with unit side-length centered at $o$, and vertices $p_1, p_2, \ldots, p_n$ in counterclockwise order.
We set $\F = \{ D_1, D_2, \ldots, D_k \}$ with $k \geq 2$, where all the $D_i$ are congruent to a disc $D$, and for $i=1,2,\ldots,k$ let $S_i = D_i \cap \partial P$. For any value of $i \neq 1$, we choose an isometry $g_{1i}: \Re^2 \to \Re^2$ satisfying $g_{1i}(D_{1}) = D_{i}$, and set $g_{i1} = g_{1i}^{-1}$, and define $g_{ij}$ by $g_{ij}(\cdot) = g_{1j}(g_{1i}^{-1} (\cdot))$ for all $i,j$.

We remark that every disc is compact, and thus, it is Lebesgue measurable. On the other hand, the boundary of a disc is not necessarily rectifiable; as an example we may choose e.g. the Koch snowflake (for more `esoteric' examples, see \cite{Sagan}). In the proof, for any disc $D$ we use the notation $\area(D)$ and $\perim(D)$ for the area and the perimeter of $D$, respectively, and we use the latter one only if $\partial D$ is clearly rectifiable. If $\Gamma$ is a rectifiable curve, then by $l(\Gamma)$ we mean the length of $\Gamma$. In particular, this yields that if $\partial D$ is rectifiable for some disc $D$, then we have $l(\partial D) = \perim(D)$.

We start with some preliminary lemmas and remarks.

\begin{remark}\label{rem:incurve}
Since any $D_i$ is a disc, any two points of $D_i$ can be connected by a continuous curve which contains only interior points of $D_i$, apart from its endpoints. In the paper, we call such a curve an \emph{in-curve} of $D_i$.
\end{remark}

\begin{remark}\label{rem:reflection}
We note that if some isometry $g_{ij}$ is a reflection about a line $L$, then $L$ separates $D_i$ and $D_j$. Indeed, suppose for contradiction that there are points $x,y \in D_i$ in different open half planes bounded by $L$, and let $\Gamma$ be an in-curve of $D_i$ connecting $x$ and $y$. Then there is a point $z$ of $\Gamma$ on $L$. Thus, $g_{ij}(z) = z$, implying that $\inter (D_i) \cap \inter (D_j) \neq \emptyset$; a contradiction. 

\end{remark}

\begin{lemma}\label{lem:diameter}
If $\diam(D) = \diam(P)$, then $k=2$. Furthermore, either $n$ is odd and $\F$ contains the two halves of $P$ dissected by a line of symmetry of $P$, or $n$ is divisible by $2$ and $\F$ is rotationally generated.
\end{lemma}

\begin{proof}
Under our conditions, each $D_i$ contains a diametrically opposite pair of points of $P$, or in other words, the two endpoints of a longest diagonal of $P$.
First, observe that if $p_{i_1}, p_{i_2} \in D_i$ and $p_{j_1}, p_{j_2} \in D_j$ are mutually distinct diametrically opposite points of $P$ where $p_{i_1}, p_{j_1}, p_{i_2}, p_{j_2}$ are in this cyclic order in $\partial P$, then any in-curve of $D_i$ connecting $p_{i_1}$ and $p_{i_2}$ would cross any in-curve of $D_j$ connecting $p_{j_1}$ and $p_{j_2}$, leading to a contradiction. Thus, there is a vertex of $P$ contained in any diameter of any $D_i$. Without loss of generality, let us assume that $p_1$ is such a vertex.

First, consider the case that $n=2m$ for some integer $m \geq 3$. Since $[p_1,p_{m+1}]$ is the unique diameter of $P$ containing $p_1$, it follows that $p_{m+1} \in D_i$ for all values of $i$. Furthermore, any isometry mapping a longest diagonal of $P$ into a longest diagonal of $P$ is a symmetry of $P$, implying that $g_{ij}$ is a symmetry of $P$ for all $i,j$. Thus, $g_{ij}$ is the reflection about the line $L$ through $[p_1,p_{m+1}]$, or about the bisector $L'$ of $[p_1,p_{m+1}]$, or about $o$. On the other hand, if $g_{ij}$ is the reflection about $L'$, the fact that $p_1$ and $p_{m+1}$ are in different open half planes bounded by $L'$ contradicts Remark~\ref{rem:reflection}.
Thus, $g_{ij}$ is the reflection about $L$ or about $o$ for any $i \neq j$. Since $g_{il}(\cdot) =  g_{jl} (g_{ij} (\cdot))$ for all $i,j,l$ by definition, the fact that there are only two possible isometries as $g_{ij}$ implies that $k \leq 2$. If $g_{12}$ is the reflection about $o$, then we are done. If $g_{12}$ is the reflection about $L$, then from Remark~\ref{rem:reflection} it follows that $L$ separates $D_1$ and $D_2$, and the tiling is rotationally generated.

Finally, consider the case that $n=2m+1$ for some integer $m \geq 2$, and let $L$ denote the line through $[o,p_1]$. By our conditions, any tile $D_i$ contains $p_{m+1}$ or $p_{m+2}$. Suppose for contradiction that a tile contains both $p_{m+1}$ and $p_{m+2}$. Then any tile contains either $p_1$, $p_{m+1}$ and $p_{m+2}$, or $p_1,p_2$ and $p_{m+2}$, or $p_1,p_n$ and $p_{m+1}$. However, this would give points $p_{i_{1}}, p_{j_{1}}, p_{i_{2}}, p_{j_{2}}$ as in the first paragraph of the proof, which was shown to be impossible.
Thus, any tile contains either $p_{m+1}$ or $p_{m+2}$.

Assume that there are at least two tiles containing one of them, say $p_{m+1} \in D_1, D_2$. Then $g_{12}$ is the reflection about either the line $L'$ through $[p_1,p_{m+1}]$, or the bisector of $[p_1,p_{m+1}]$, or the midpoint of $[p_1,p_{m+1}]$. Here the second case contradicts Remark~\ref{rem:reflection}. In the first and the third cases we have that $D_1, D_2 \subset P \cap P'$, where $P'=g_{12}(P)$ (cf. Figure~\ref{fig:lem_diameter}). Thus, $P' \cap P \subsetneq P$ yields that $k \geq 3$. If $k=3$, then $P \setminus P' \subseteq D_3$, and the compactness of $D_3$ implies that $p_{m+1}, p_{m+2}\in D_3$; a contradiction, as in the previous paragraph. If $k > 3$, then there are at least two tiles containing $p_{m+2}$, which, by the previous argument, are contained in $P \cap P''$, where $P''$ is the reflected copy of $P$ about the line through $[p_1,p_{m+2}]$. Thus, in this case the midpoint of $[p_{m+1},p_{m+2}]$ does not belong to any tile; a contradiction.

\begin{figure}[ht]
\begin{center}
 \includegraphics[width=0.38\textwidth]{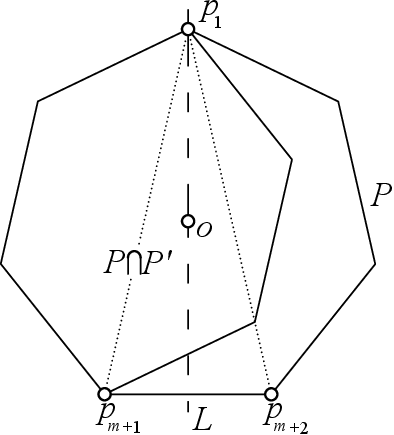}
 \caption{An illustration for the proof of Lemma~\ref{lem:diameter}}
\label{fig:lem_diameter}
\end{center}
\end{figure}

We have shown that $k=2$, and there is a unique tile containing $p_{m+1}$ and a unique tile containing $p_{m+2}$. Let these tiles be $D_1$ and $D_2$, respectively.
Let $q$ be the midpoint of $[p_{m+1},p_{m+2}]$ and assume, without loss of generality, that $q \in D_1$. Then the only congruent copy of $P$ containing $p_1,p_{m+1},q$ is $P$, implying that $g_{12}(P) = P$. Since we also have $g_{12}([p_1,p_{m+1}]) =[p_1,p_{m+2}]$, this yields that $g_{12}$ is the reflection about the line $L$ through $[o,p_1]$, from which the assertion readily follows.
\end{proof}

Next, we recall Lemma 2.3 from \cite{KLV2020}.

\begin{lemma}\label{lem:topology3}
Let $\{ \bar{D}_{1}, \bar{D}_{2}, \bar{D}_{3}\}$ be a tiling of the disc $\bar{D}$ where, for $i= 1,2,3$, $\bar{D}_{i}$ is a disc such that $\bar{S}_{i} = \bar{D}_{i} \cap \partial \bar{D}$ is a nondegenerate simple continuous curve. Then $\bar{D}_{1} \cap \bar{D}_{2} \cap \bar{D}_{3}$ is a singleton $\{q\}$, and for any $i \neq j, \bar{D}_{i} \cap \bar{D}_{j}$ is a simple continuous curve connecting $q$ and a point in $\partial \bar{D}$.
\end{lemma} 

Our next lemma is a generalization of Lemma~\ref{lem:topology3}.

\begin{lemma}\label{lem:topology}
Let the disc $\bar{D}$ be decomposed into three discs $\bar{D}_1$, $\bar{D}_2$ and $\bar{D}_3$. For $i=1,2,3$, set $\bar{S}_i = \bar{D}_i \cap \partial \bar{D}$.
Then, with a suitable choice of indices, exactly one of the following holds (cf. Figure~\ref{fig:types}).
\begin{enumerate}
\item[(1)] $\bar{S}_3$ contains at most two points, and $\bar{S}_1$ and $\bar{S}_2$ are connected arcs whose union covers $\partial \bar{D}$.
\item[(2)] $\bar{S}_1$ is the union of two disjoint, connected, nondegenerate arcs, the sets $\bar{S}_2, \bar{S}_3, \bar{D}_1 \cap \bar{D}_2$, $\bar{D}_1 \cap \bar{D}_3$ are connected arcs, and $\bar{D}_2 \cap \bar{D}_3 = \emptyset$.
\item[(3)] $\bar{S}_2$, $\bar{S}_3$, $\bar{D}_1 \cap \bar{D}_2$, $\bar{D}_1 \cap \bar{D}_3$, $\bar{D}_2 \cap \bar{D}_3$ are connected arcs, $\bar{D}_1 \cap \bar{D}_2 \cap \bar{D}_3$ is a singleton $\{ q \}$, and $\bar{S}_1$ is either a connected arc, or the union of a connected arc and $\{ q\}$.
\end{enumerate}
\end{lemma}

\begin{figure}[ht]
\begin{center}
 \includegraphics[width=0.97\textwidth]{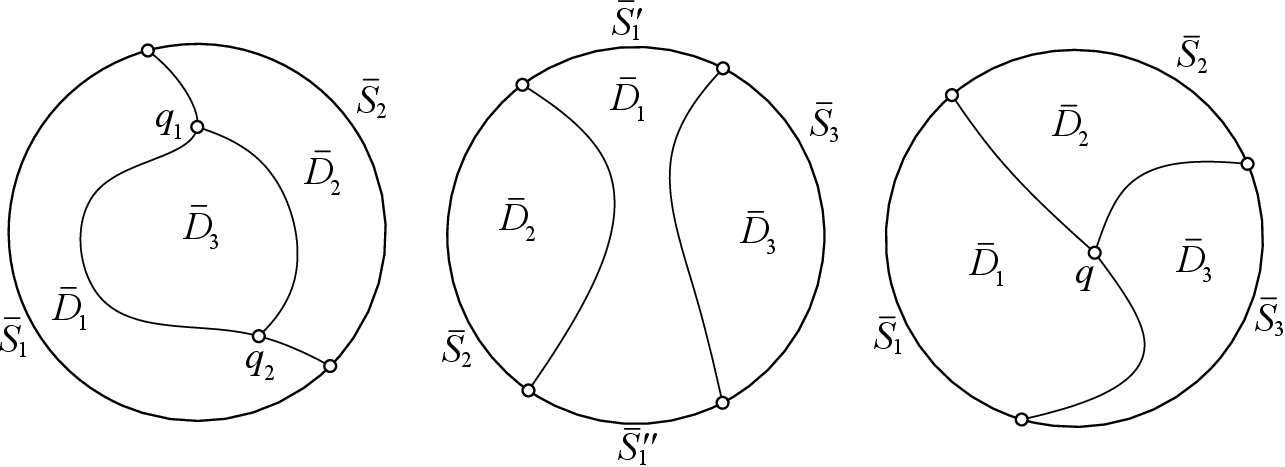}
 \caption{The topological types described in Lemma~\ref{lem:topology} with a Euclidean disc as $\bar{D}$. We note that the points $q_1$ and $q_2$ in the left panel, and $q$ in the right panel may lie on $\partial \bar{D}$.}
\label{fig:types}
\end{center}
\end{figure}

\begin{proof}
Assume that one of the $\bar{S}_i$, say $\bar{S}_1$, has more than one component, and let $q_1, q_2, \ldots, q_m$ be points of $\bar{S}_1$, in this cyclic order, contained in different components of $\bar{S}_1$. Let $r_1, r_2, \ldots, r_m$ be points in $ {\partial \bar{D}} \setminus \bar{S}_1$ such that $q_1, r_1, q_2, \ldots, q_m, r_m$ are in this cyclic order in ${\partial \bar{D}}$. Note that every $r_j$ belongs to $\bar{S}_2$ or $\bar{S}_3$, and no two of them belongs to the same set. Indeed, if, say, $r_{j_1}$ and $r_{j_2}$ belong to $\bar{S}_2$, where $j_1 \neq j_2$, then any in-curve in $\bar{D}_2$ connecting them, and any in-curve in $\bar{S}_1$ connecting $q_{j_1}$ and $q_{j_2}$ would cross, which is a contradiction.
Thus, we have $m = 2$, which also yields, by the same argument, that $\bar{S}_2$ and $\bar{S}_3$ are connected. If neither component of $\bar{S}_1$ is a singleton, then the closure of $(\partial \bar{D}_1) \setminus \bar{S}_1$ contains two disjoint, simple curves which, apart from their endpoints, are contained in $\inter(\bar{D})$. Since in this case $\bar{D}_2$ and $\bar{D}_3$ can be separated by an in-curve of $\bar{D}_1$ disjoint from $\bar{D}_2 \cup \bar{D}_3$, it follows by compactness that the two components of $\cl ((\partial \bar{D}_1) \setminus \bar{S}_1)$ coincide with $\bar{D}_1 \cap \bar{D}_2$ and $\bar{D}_1 \cap \bar{D}_3$, implying (2). If exactly one component of $\bar{S}_1$ is a singleton, a similar argument can be applied, implying (3). Finally, if both components of $\bar{S}_1$ are singletons, the conditions in (1) are satisfied with a suitable relabeling of the tiles.

Assume that all the $\bar{S}_i$ are connected. Since $\bar{S}_1 \cup \bar{S}_2 \cup \bar{S}_3 = \partial \bar{D}$ and every $\bar{S}_i$ is a simple connected arc properly contained in $\partial \bar{D}_i$, we have that at least two of the $\bar{S}_i$ contain more than one point. If one of them, say $\bar{S}_3$, contains at most one point, then (1) follows. If $\bar{S}_1$, $\bar{S}_2$ and $\bar{S}_3$ are nondegenerate, simple arcs, then the conditions of Lemma~\ref{lem:topology3} are satisfied, implying (3).
\end{proof}

\begin{definition}\label{defn:types}
Let $\{ \bar{D}_{1}, \bar{D}_{2}, \bar{D}_{3}\}$ be a tiling of the disc $\bar{D}$. If the discs $\bar{D}_1, \bar{D}_2, \bar{D}_3$ satisfy the conditions in ($k$) of Lemma~\ref{lem:topology} with $k=1$, $k=2$ or $k=3$, we say that the tiling is a \emph{Type $k$ decomposition} of $\bar{D}$. 
\end{definition}

\begin{remark}\label{rem:topfor2}
Assume that $\bar{D}$ is decomposed into two discs $\bar{D}_1, \bar{D}_2$, and for $i=1,2$, set $\bar{S}_i = \bar{D}_i \cap \partial \bar{D}$. Note that since the number of tiles is more than one, we have that no $\bar{S}_i$ coincides with $\partial \bar{D}$. Furthermore, by the argument in the proof of Lemma~\ref{lem:topology} we also have that $\bar{S}_1$ and $\bar{S}_2$ are connected. Motivated by this property, we call any tiling of $\bar{D}$ with two discs a Type 1 decomposition of $\bar{D}$.
\end{remark}

Lemmas~\ref{lem:finite_copies}-\ref{lem:cutcong} and Definitions~\ref{defn:multicurve}-\ref{defn:equidecomposable} are from \cite{KLV2020}.	
		
\begin{lemma}\label{lem:finite_copies}
Let $G$ and $C$ be simple curves. Then $G$ contains at most finitely many congruent copies of $C$ which are mutually disjoint, apart from possibly their endpoints.
\end{lemma}

\begin{definition}\label{defn:multicurve}
 A \emph{multicurve}  (see also \cite{KurusaMF}) is
 a finite family of simple curves,
 called the \emph{members of the multicurve},
 which are parameterized on nondegenerate closed finite intervals,
 and any point of the plane belongs to at most one member,
 or it is the endpoint of exactly two members.
 If $\mathcal{F}$ and $\mathcal{G}$ are multicurves,
 $\bigcup \mathcal{F}= \bigcup \mathcal{G}$,
 and every member of $\mathcal{F}$ is the union of some members of $\mathcal{G}$,
 we say that $\mathcal{G}$ is a \emph{partition} of $\mathcal{F}$.
\end{definition}

\begin{definition}\label{defn:equidecomposable}
 Let $\mathcal{F}$ and $\mathcal{G}$ be multicurves.
 If there are partitions $\mathcal{F}'$ and $\mathcal{G}'$ of $\mathcal{F}$
 and $\mathcal{G}$, respectively, and a bijection
 $f\colon\mathcal{F}' \to \mathcal{G'}$ such that $f(C) \in \mathcal{G}'$ is
 congruent to $C$ for all $C \in \mathcal{F}'$,
 we say that $\mathcal{F}$ and $\mathcal{G}$ are \emph{equidecomposable}.
\end{definition}

\begin{lemma}\label{lem:graphs}
 If $\mathcal{F}$ and $\mathcal{G}$ are multicurves
 with $\bigcup \mathcal{F} = \bigcup \mathcal{G}$,
 then $\mathcal{F}$ and $\mathcal{G}$ are equidecomposable.
\end{lemma}

\begin{lemma}\label{lem:cutcong}
If $\mathcal{F}$ and $\mathcal{G}$ are equidecomposable,
and their subfamilies $\mathcal{F}' \subseteq \mathcal{F}$
and $\mathcal{G}' \subseteq \mathcal{G}$ are equidecomposable,
then $\mathcal{F} \setminus \mathcal{F}'$ and $\mathcal{G} \setminus \mathcal{G}'$
are equidecomposable.
\end{lemma}

We finish with a remark and a definition.

\begin{remark}\label{rem:overlapping}
Let $\{ D_1, D_2, D_3 \}$ be a monohedral tiling of the regular $n$-gon $P$ with unit side-length and $n \geq 5$. For $i=1,2,3$, set $S_i = D_i \cap \partial P$.
Note that for any $i \neq j$, $g_{ij}(S_i) \subset (\partial \conv (D_j)) \cap ( \partial D_j)$. Furthermore, we have the following:
\begin{enumerate}
\item[(1)] If $S_i^* \subseteq S_i$ and $S_j^* \subseteq S_j$ are maximal nondegenerate segments in $S_i$ and $S_j$, respectively, such that the interiors of $S_i^*$ and $g_{ji}(S_j^*)$ intersect, then $S_i^* = g_{ji}(S_j^*)$.
\item[(2)] If some vertex $p_t$ of $P$ lies in the interiors of both $S_i$ and $g_{ji}(S_j)$, then $S_i = g_{ji}(S_j)$, and $P=g_{ji}(P)$.
\item[(3)] If $S_i \cap g_{ji}(S_j)$ contains a segment of unit length, then  $S_i = g_{ji}(S_j)$, and $P=g_{ji}(P)$.
\end{enumerate}
If the interiors of $S_i$ and $g_{ji}(S_j)$ are disjoint, we say that $S_i$ and $g_{ji}(S_j)$ are \emph{nonoverlapping}.
Based on our above observations and Lemma~\ref{lem:topology}, if $S_i$ and $g_{ji}(S_j)$ overlap but they do not coincide, then their intersection contains at most two connected components, each of which is either a single point or a nondegenerate segment of length strictly less than one. In this case we say that $S_i$ and $g_{ji}(S_j)$ are \emph{slightly overlapping} (cf. Figure~\ref{fig:overlap}).
We observe that $S_i$ and $S_j = g_{ij}(S_i)$ are nonoverlapping, slightly overlapping and equal if and only if $g_{il} (S_i)$ and $g_{jl}(S_j)$ are nonoverlapping, slightly overlapping and equal, respectively, for an arbitrary value of $l$.

\end{remark}

\begin{figure}[ht]
\begin{center}
 \includegraphics[width=0.9\textwidth]{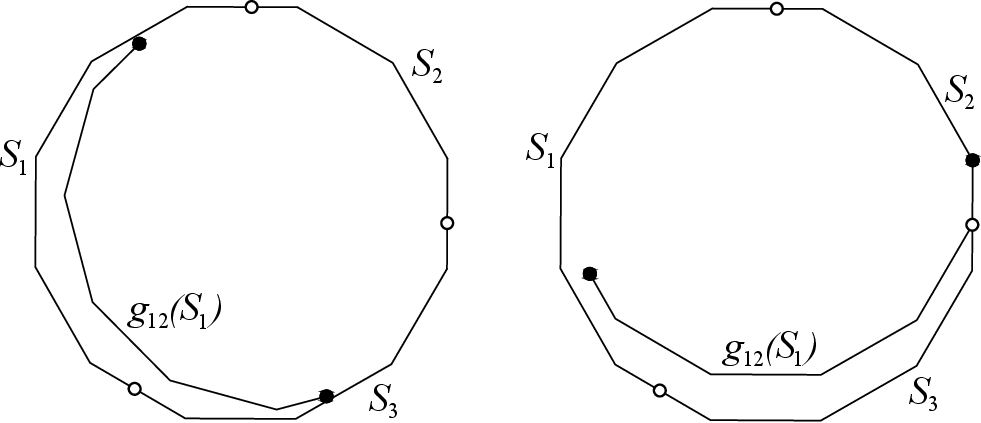}
 \caption{An illustration for Remark~\ref{rem:overlapping}. In the left panel $S_2$ and $g_{12}(S_1)$ do not overlap; in the right panel $S_2$ and $g_{12}(S_1)$ slightly overlap. The endpoints of the arcs $S_i$ are denoted by empty circles, and the endpoints of $g_{12}(S_1)$ are denoted by full circles.}
\label{fig:overlap}
\end{center}
\end{figure}


\section{Proof of Theorem~\ref{thm:main}}\label{sec:proof}

By Lemma~\ref{lem:topology} and Remark~\ref{rem:topfor2}, it is sufficient to prove the theorem for Type 1, Type 2 and Type 3 decompositions of $P$.
In the following, we present the proof for each type in a separate subsection. Throughout the proof, we assume that no $D_i$ contains diametrically opposite points of $P$, as otherwise the assertion readily follows from Lemma~\ref{lem:diameter}.

\subsection{Proof for Type $1$ decompositions}\label{subsec1}

We choose our notation in such a way that $S_1$ and $S_2$ are nondegenerate connected arcs in $\partial P$ whose union is $P$, and $S_3$, if it exists, contains at most two points. Observe that in this case $n$ is even, as otherwise $S_1$ or $S_2$ contains at least $\frac{n+1}{2}$ vertices of $P$, including a pair of diametrically opposite points of $P$. Consider the sets $S_1$ and $S_1'=g_{21}(S_2)$. By Remark~\ref{rem:overlapping}, we distinguish three cases.

\smallskip

\emph{Case 1}, $S_1$ and $S_1'$ do not overlap.\\
In this case they are nonoverlapping arcs in the boundary of $\conv (D_1)$ whose total length is $\perim(P)$, which implies that $\conv(D_1)=P$ and $\partial (\conv D_1) = S_1 \cup S_1'$. On the other hand, since in this case $S_1 \cup S_1'$ is a simple, closed curve, we have $D_1 = \conv (D_1)=P$, which contradicts the assumption that $k > 1$.

\smallskip

\emph{Case 2}, $S_1$ and $S_1'$ slightly overlap.\\
Let $L$ be a sideline of $P$ containing at least an interior point of $S_1' \cap S_1$, and let $L'$ be the supporting line of $P$ parallel to $L$ which is different from $L$. Let $G_1$ and $G_2$ denote the two components of $\cl (\partial P \setminus (L \cup L'))$. Clearly, since $L$ contains a common endpoint of $S_1$ and $S_2$, at least one of $G_1$ and $G_2$ contains no endpoint of $S_1$ and $S_2$ in its interior, and hence, we may assume that e.g. $G_2 \subset S_2$. Then the facts that $S_1'$ and $S_1$ are slightly overlapping and $S_1' \subset P$ yield that $L'$ also contains a point of $S_1'$, and $G_1 \subset S_1$. Furthermore, since in this case $S_1 \cup S_1'$ form simple closed curve(s) in $\partial D_1$, we have that $D_1 = \conv(S_1 \cup S_1')$. Let $q$ and $q'$ be the midpoints of the sides of $P$ on $L$ and $L'$, respectively, and observe that the translate $G_2'$ of $G_2$ whose endpoints are $q$ and $q'$ is contained in $D_1$. Thus, the area of $D_1$ is greater than or equal to the area of $\conv (G_1 \cup G_2')$, and the area of the latter region is strictly greater than $\frac{\area(P)}{2}$. This contradicts the fact that the examined tiling of $P$ is monohedral.

\smallskip

\emph{Case 3}, $S_1 = S_1'$.\\
In this case $g_{21}$ is either the reflection about the line $L$ through the two endpoints of $S_1$ and $S_2$, or it is the reflection about $o$. This implies Theorem~\ref{thm:main} for $k=2$. Furthermore, if $k=3$, then $D_3 = \cl (P \setminus (D_1 \cup D_2))$ is symmetric to $L$ or $o$, respectively, and $P$ has an even number of sides.

Assume that $k=3$ and $g_{21}$ is the reflection about $o$. Then, since $D_1$, $D_2$ and $D_3$ are all congruent, it follows that both $D_1$ and $D_2$ are centrally symmetric. As $D_1, D_2 \subset P$ we also have that the centers of $D_1$ and $D_2$ are contained on the line through $o$ parallel to the two sides of $P$ containing the common endpoints of $S_1$ and $S_2$. Let these two sides of $P$ be $E$ and $E'$, and let the centers of symmetry of $D_1$ and $D_2$ be $c_1$ and $c_2$, respectively. From the properties of central symmetry and the fact that $S_1 = S_1'$, we have that for $i=1,2$, $E \cap S_i$ and $E' \cap S_i$ are segments of length $1/2$. Furthermore, for $i=1,2$ the union of $S_i$ and its reflection about $c_i$ is a simple closed convex curve in $\partial D_i$, implying that its convex hull is $D_i$. Thus, $D_1$ and $D_2$ overlap; a contradiction.

Finally, assume that $k=3$ and $g_{21}$ is the reflection about the line $L$ passing through the common endpoints of $S_1$ and $S_2$. Since $L$ is a symmetry line of $P$ and no $D_i$ contains diametrically opposite points of $P$, it follows that $L$ passes through the midpoints of two opposite sides $E, E'$ of $P$.
Furthermore, since $D_3$ is symmetric to the line $L$, we obtain that $D_1$ and $D_2$ have lines of symmetry, which we denote by $L_1$ and $L_2$, respectively. Since both discs are contained in the infinite strip bounded by the two sidelines of $P$ through $E$ and $E'$, $L_1$ and $L_2$ are parallel to $L$, or coincide with the line $L^*$ through $o$ perpendicular to $L$.

Assume that one of $L_1$ and $L_2$, say $L_1$, is parallel to $L$, and let $S'$ denote the reflection of $S_1$ about $L_1$. Since $S_1 \cup S'$ is a simple, closed curve in $\partial D_1$, we have $S_1 \cup S' = \partial D_1$, which yields that $D_1 = \conv (S_1 \cup S')$. On the other hand, as the endpoints of $S_1$ are midpoints of two opposite sides of $P$, from this an elementary computation shows that $\area(D_1) > \frac{\area(P)}{3}$, a contradiction. Thus, we have $L_1 = L^*$, and we remark that our argument shows that \emph{any} line of symmetry of $D_1$ coincides with $L^*$, and, applying this argument, we obtain the same statement for $D_2$. On the other hand, if both $D_1$ and $D_2$ are symmetric to $L^*$, then the same holds for $D_3$. Hence, $D_3$ is symmetric to both $L$ and $L^*$, which yields that $D_1$ (resp., $D_2$) has a line of symmetry different from $L^*$, which contradicts our previous observation.

\subsection{Proof for Type $2$ decompositions}\label{subsec2}

We assume that $S_1$ is disconnected, and denote the two components of $S_1$ by $S_1'$ and $S_1''$.
We distinguish three cases.

\emph{Case 1}, both $S_1'$ and $S_1''$ contain vertices of $P$.\\
Without loss of generality, we may assume that the vertices of $P$ in $S_1'$ are $p_1, p_2, \ldots, p_m$ for some $1 \leq m \leq n-1$.
First, we show that $n$ is even. Suppose for contradiction that $n= 2t+1$ for some $t \geq 2$.
Then, since $S_1$ contains no diametrically opposite points of $P$, we have that 
$p_{t+1}, p_{t+2}, \ldots, p_{t+m+1}$ belong to the same set $S_i \neq S_1$. Without loss of generality, we may assume that they, and also $p_{m+1}$, belong to $S_2$.  
This yields that $p_{m+1},p_{m+2}, \ldots, p_{m+t+1}$ belong to $S_2$, and thus, $S_2$ contains at least $t+1$ vertices of $P$, which contradicts our assumption that $D_2$ does not contain diametrically opposite points of $P$. Thus, we have that $n$ is even.

Let $n=2t$ for some $t \geq 3$. Similarly like in the previous paragraph, we have that $p_{t+1}, \ldots, p_{t+m}$ are not points of $S_1''$, and thus, they all belong to $S_2$ or all belong to $S_3$. Without loss of generality, assume that they belong to $S_2$. Thus, $p_{m+1}, \ldots, p_{m+t} \in S_2$. Since no $S_i$ contains diametrically opposite points of $P$, we also have $m \leq t$ and $p_{m+t+1} \in S_1''$. This implies that $l(S_3) < l(S_2)$. Since neither $S_1$ nor $S_2$ contains diametrically opposite vertices of $P$, we also have that the endpoints of $S_2$ are interior points of two sides of $P$. The facts that $S_2$ contains exactly $t$ vertices of $P$ and $S_1$ is disconnected yield also that $S_2 \neq g_{12}(S_1)$ and $S_2 \neq g_{32}(S_3)$. Furthermore, $g_{32}(S_3)$ and $S_2$ are not slightly overlapping, since otherwise $S_3$ and $g_{23}(S_2)$ are slightly overlapping (cf. Remark~\ref{rem:overlapping}), which contradicts the fact that in this case the side of $P$ opposite the overlap is contained in $S_2$. By a similar argument, $S_{3}$ and $g_{13}( S_{1})$ do not slightly overlap. Thus, we have that either $S_1$, $g_{21}(S_2)$ and $g_{31}(S_3)$ are mutually nonoverlapping, or $S_1$ and $g_{21}(S_2)$ slightly overlap.

If $S_1$, $g_{21}(S_2)$ and $g_{31}(S_3)$ are mutually nonoverlapping, then their total length is equal to $\perim(P)$, implying that $\perim(\conv(D_1)) \geq \perim(P)$. This yields that $D_1 = \conv (D_1) = P$, which contradicts our assumption that $k = 3$ for any Type 2 decomposition of $P$. Thus, the only possibility left is that $S_1$ and $g_{21}(S_2)$ slightly overlap. Since the endpoints of $S_2$ are interior points of two opposite sides of $P$, this yields that $g_{21}(S_2)$ contains at least one endpoint of $S_2$. On the other hand, since apart from its endpoints, no point of $S_1$ may belong to $D_2$, we also have that $g_{21}(S_2)$ contains both endpoints of $S_2$, and also that it is a translate of $S_2$. Let the endpoints of $S_2$ on $[p_m,p_{m+1}]$ and $[p_{m+t},p_{m+t+1}]$ be $q$ and $q'$, respectively. Then $g_{21}([p_{m+1},q])$ is either $[p_m,q]$ or $[p_{m+t+1},q']$, which yields that $q$ is the midpoint of $[p_m,p_{m+1}]$ and $q'$ is the midpoint of $[p_{m+t},p_{m+t+1}]$ (cf. Figure~\ref{fig:type2_case1}). On the other hand, since $g_{21}(S_2) \subset \partial D_1$, it separates $D_2$ from $D_1 \cup D_3$. In other words, $D_2$ is the region bounded by the union of $S_2$ and the part of $g_{21}(S_2)$ connecting $q$ and $q'$. But this and the fact $S_3$ is not empty yields that $D_1$ is the translate of $D_2$ by the vector $q-p_{m+1}$, and hence, $D_3 = \cl (P \setminus (D_1 \cup D_2))$ is not congruent to $D_1$ and $D_2$; a contradiction.

\begin{figure}[ht]
\begin{center}
 \includegraphics[width=0.4\textwidth]{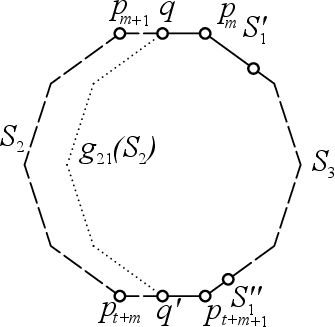}
 \caption{An illustration for Case 1 in Subsection~\ref{subsec2}, where $t=4$. In the picture $S_1'$ and $S_1''$ are drawn with solid, $S_2$ and $S_3$ with dashed, and $g_{21}(S_2)$ with dotted lines.}
\label{fig:type2_case1}
\end{center}
\end{figure}

\emph{Case 2}, exactly one of $S_1'$ and $S_1''$ contains a vertex of $P$.\\
Let the vertices of $P$ in $S_1'$ be $p_1, p_2, \ldots, p_m$ for some $1 \leq m < \frac{n+1}{2}$.

First, we consider the case that $n$ is odd, namely that $n=2t+1$ for some integer $t \geq 2$.
Observe that since the diameter graph of the vertex set of $P$ is an odd cycle, and hence, it cannot be colored with two colors, every $S_i$ contains a vertex of $P$ in its interior.
We show that the side of $P$ containing $S_1''$ is opposite a vertex in $S_1'$. Indeed, suppose for contradiction that $[p_{t+1},p_{t+2}], \ldots, [p_{t+m},p_{t+m+1}]$ are disjoint from $S_1''$. Then they all belong to the same $S_i$, and thus, we may assume that $p_{m+1}, \ldots, p_{t+m+1}$ belong to $S_2$. But then $S_2$ contains diametrically opposite vertices of $P$, which contradicts our assumption. Hence, we have that for some $1 \leq i \leq m$,  $S_1'' \subset [p_{i+t},p_{i+t+1}]$.

Note the fact that $S_1$ is disconnected implies that $g_{1i}(S_1)$ does not coincide with $S_i$. Assume that, say, $g_{12}(S_1)$ slightly overlaps $S_2$. Then $g_{12}$ is the composition of a symmetry of $P$ and a (nondenegerate) translation parallel to a side of $P$, which contradicts the fact that $g_{12}(S_1) \subset P$.
Thus, we have that $S_1$ does not overlap $g_{21}(S_2)$ and $g_{31}(S_3)$. Without loss of generality, we may assume that $l(S_2) \geq l(S_3)$.
For $i=2,3$, let $T_i$ denote the segment connecting the endpoints of $S_i$, $K_i = \conv(S_i)$ and $C_i = \conv(D_1) \cap K_i$ (cf. Figure~\ref{fig:type2_case2}).
Recall that $g_{21}(S_2)$ and $g_{31}(S_3)$ are contained in $\partial (\conv (D_1))$ and they do not overlap $S_1$. Since $S_2$ and $S_3$ contain vertices of $P$, we also have that they do not overlap $T_2$ and $T_3$. Thus, in particular, $C_2$ or $C_3$ is a plane convex body with perimeter at least $l(T_2) + l(S_2)$ or $l(T_3) + l(S_2)$, respectively. As $\perim(K_i) = l(S_i) + l(T_i)$ for $i=2,3$, this yields that $g_{21}(S_2)$ coincides with $S_2$ or $S_3$, a contradiction. 

\begin{figure}[ht]
\begin{center}
 \includegraphics[width=0.37\textwidth]{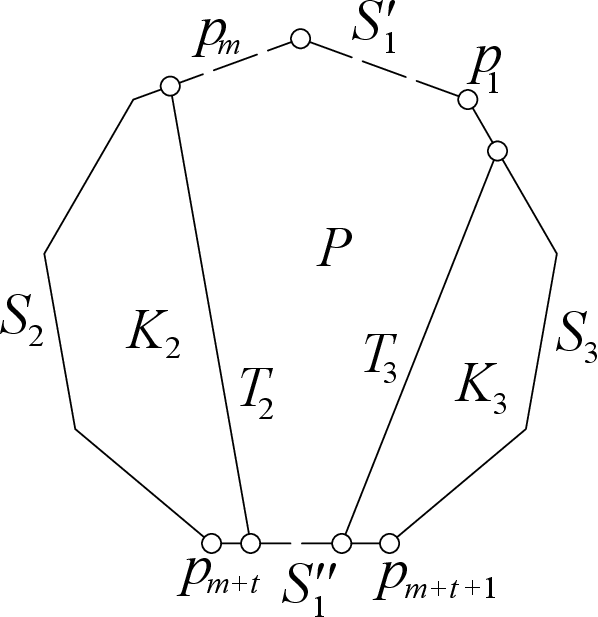}
 \caption{An illustration for Case 2 in Subsection~\ref{subsec2}, where $t=4$ and $m=2$. In the picture $S_1'$ and $S_1''$ are drawn with dashed, and $S_2$ and $S_3$ are drawn with solid lines.}
\label{fig:type2_case2}
\end{center}
\end{figure}

In the remaining part of Case 2, we assume that $n$ is even, i.e. $n=2t$ for some $t \geq 3$. Assume that $S_1', S_2, S_1'', S_3$ are in counterclockwise direction on $\partial P$. Note that at least one of $S_2$ and $S_3$ contains vertices of $P$. Furthermore, similarly to the $n$ odd case, the fact that neither $S_2$ nor $S_3$ contains diametrically opposite points yields that there are points $q' \in S_1'$ and $q'' \in S_1''$ on opposite sides of $P$.

Clearly, $g_{12}(S_1) \neq S_2$ as $S_1$ is disconnected. Assume that they are slightly overlapping. Then $g_{12}(q')$ and $ g_{12}(q'')$ lie on opposite sides of $P$, and hence, they belong to $S_2$. Thus, $S_2$ contains $t$ vertices of $P$, and $q',q''$ and their images under $g_{12}$ are on the same two sides of $P$. This yields that $g_{12}$ is either a translation parallel to these sides, or the composition of such a translation with a reflection to a line parallel or perpendicular to these sides, or the origin. Let these two sides be $E'$ and $E''$ with $q' \in E'$ and $q'' \in E''$. Then $ g_{12}((E' \cup E'')\cap S_1) = (E' \cup E'')\cap S_2$. Since $E' \cap S_1$, $E' \cap S_2$, $E'' \cap S_2$ and $E'' \cap S_1$ are in counterclockwise order on $\partial P$, we have that  $g_{12}$ is a translation parallel to $E'$, or its composition with the reflection about the line through $o$ and parallel to $E'$. But both cases contradict the fact that $E'' \cap (S_1 \cup S_2)$ is strictly shorter than $E' \cap (S_1 \cup S_2) = E'$.

We have obtained that $g_{12}(S_1)$ and $S_2$ do not overlap. It follows similarly that $g_{13}(S_1)$ and $S_3$ do not overlap, or equivalently, that $g_{21}(S_2)$ and $g_{31}(S_3)$ do not overlap $S_1$. But in this case we may apply the argument used for $n$ odd.

\emph{Case 3}, neither $S'_1$ nor $S''_1$ contains a vertex of $P$.\\
Recall that by the definition of Type 2 configuration (cf. Lemma~\ref{lem:topology}), both $S_1'$ and $S_1''$ are nondegenerate segments.

As we remarked in Case 2, if $n$ is odd, then the diameter graph of the vertex set of $P$ is an odd cycle, which is not $2$-colorable. But this contradicts the assumptions that none of $S_1$, $S_2$, $S_3$ contains diametrically opposite vertices of $P$, and $S_1$ does not contain a vertex of $P$. Thus, the condition of Case 3 is satisfied only if $n$ is even, and $S_1'$ and $S_1''$ lie on opposite sides of $P$. Let these sides of $P$ be $E'$ and $E''$ with $S_1' \subset E'$.

Clearly, we have $g_{21}(S_2) \neq S_1$. Consider the case that $g_{21}(S_2)$ slightly overlaps $S_1$. Then the endpoints of $g_{21}(S_2)$ lie on $E'$ and $E''$, which yields that either $D_2$ is the region bounded by $S_2 \cup \left( g_{21}(S_2) \setminus (E' \cup E'') \right)$, or $D_3$ is the region bounded by $S_3\cup \left( g_{21}(S_2) \setminus g_{21} (E' \cup E'') \right)$. From these two cases we obtain $\area(D_2) < \frac{1}{3} \area(P)$ and $\area(D_3) < \frac{1}{3} \area(P)$, respectively, which contradicts the fact that the tiling is monohedral. Thus, we are left with the case that $g_{21}(S_2)$ and $S_1$ do not overlap. Similarly, we obtain that $g_{31}(S_3)$ and $S_1$ do not overlap. But then $S_1$, $g_{21}(S_2)$ and $g_{31}(S_3)$ are mutually nonoverlapping arcs in $\partial (\conv (D_1))$ with total length equal to $\perim(P)$, implying that $D_1 = \conv(D_1) = P$; a contradiction.

\subsection{Proof for Type $3$ decompositions}\label{subsec:type3}

The proof presented in this subsection roughly follows the structure of the proof in \cite{KLV2020} with some of the arguments borrowed from there; in particular, depending on the number of coinciding arcs among $S_1, g_{21}(S_2)$ and $g_{31}(S_3)$, we distinguish three cases.

\emph{Case 1}, no two of $S_1, g_{21}(S_2)$ and $g_{31}(S_3)$ coincide.\\
If no two of these arcs overlap, then the fact that their total length is equal to $\perim(P)$ yields that $D_1 = \conv (D_1) = P$; a contradiction.
Thus, we have that at least one pair among them overlaps. Before proceeding further, we use this observation to show that all of $S_1$, $S_2$ and $S_3$ contain at least one vertex of $P$ different from their endpoints. Indeed, suppose for contradiction that one of them, say $S_1$, contains no vertex of $P$ in its interior. Then, since $D_1, D_2$ and $D_3$ contain no diametrically opposite points of $P$, we have that $n$ is even, and that both $S_2$ and $S_3$ contain a vertex of the side of $P$ that contains $S_{1}$ and interior points of the opposite side. Thus, if $g_{21}(S_2)$ and $S_1$ slightly overlap, then $S_1$ intersects a pair of opposite edges of $P$ in nondegenerate segments and the configuration is not Type $3$; a contradiction. The cases that $g_{31}(S_3)$ slightly overlaps $S_1$ or $g_{21}(S_2)$ can be eliminated by a similar argument. Thus, we obtain that $S_1$, $g_{21}(S_2)$ and $g_{31}(S_3)$ do not overlap; a contradiction. Hence, in the remaining part of Case 1 we assume that all of $S_1$, $S_2$ and $S_3$ contain vertices of $P$ different from their endpoints.

Consider the case that there are at least two overlapping pairs among $S_1, g_{21}(S_2)$ and $g_{31}(S_3)$; without loss of generality, we may assume that $g_{21}(S_2)$ and $g_{31}(S_3)$ slightly overlap $S_1$. Then, for $i=2,3$, $g_{1i}(S_1)$ slightly overlaps $S_i$, which, combining it with the fact that the angles of $P$ are obtuse, implies that two of $\partial D_1, \partial D_2$ and $\partial D_3$ cross; a contradiction.

We are left with the case that exactly one pair of $S_1, g_{21}(S_2)$ and $g_{31}(S_3)$ overlaps. Without loss of generality, we may assume that $S_1$ and $g_{21}(S_2)$ overlap. Let us choose our notation such that the vertices of $P$ on $S_1$ are $p_1, p_2, \ldots, p_m$, and $S_1, S_2, S_3$ are in counterclockwise order around $P$. For any $i \neq j$, let $q_{ij}$ be the common endpoint of $S_i$ and $S_j$. We assume that $q_{23} \in [p_l,p_{l+1}]$ with $q_{23} \neq p_{l+1}$. By our assumption,
we have that $g_{21}(q_{12})$ or $g_{21}(q_{23})$ lies in the interior of $S_1$. Depending on which one of $q_{12}$ and $q_{23}$ lies on which side of $S_1$, we distinguish four cases. 

If $g_{21}(q_{12})$ lies in the interior of $S_1$ and $g_{21}(q_{12}) \in [p_m,p_{m+1}]$, then $q_{12}$ is the midpoint of $[p_m,p_{m+1}]$ and $g_{21}([q_{12},p_{m+1}]) = [q_{12},p_m]$. This implies that $\partial D_1$ and $\partial D_2$ cross at $q_{12}$; a contradiction.

Assume that $g_{21}(q_{12})$ lies in the interior of $S_1$ and $g_{21}(q_{12}) \in [p_n,p_1]$. Then $g_{21}(q_{12}) = p_1$, $g_{21}(p_{m+1})=q_{13}$ (or equivalently, $g_{12}(p_1)=q_{12}$ and $g_{12}(q_{13})=p_{m+1}$), and $q_{12}$ and $q_{13}$ are interior points of $[p_m,p_{m+1}]$ and $[p_n,p_1]$, respectively. Note that $\conv (D_2) \subset P$ implies that $\conv(D_1) \subset g_{21}(P)$. On the other hand, $g_{21}(P)$ is the translate of $P$ by the vector $q_{13}-p_n$. Let $P_0 = P \cap (q_{12}-p_n + P)$ (cf. Figure~\ref{fig:type3_case1}). Then $P_0$ is a convex polygon such that each one of its sides is parallel to some side of $P$. Let $C=S_1 \cup g_{21}(S_2)$, and observe that $C$ and $g_{31}(S_3)$ are non-overlapping curves in $\partial (\conv(D_1))$.
If $q_{23}$ is not a vertex of $P$, then the total turning angle along the curves $C$ and $g_{32}(S_3)$ is $2\pi$, which implies that $\partial (\conv (D_1))$ is the union of $C$ and $g_{31}(S_3)$ and possibly two segments such that the lines through them contain segments from both $C$ and $g_{32}(S_3)$; this contradicts the fact that $\conv(D_1)$ is contained in $P_0$. If $q_{23}$ is a vertex of $P$, we may apply the same argument after observing that $\conv (C \cup g_{31}(S_3))$ is a convex polygon, and the fact that $\conv (C \cup g_{31}(S_3)) \subset \conv(D_1) \subset P_0$ implies that the turning angle of $\conv (C \cup g_{31}(S_3))$ at $g_{21}(q_{23})$ is at least $\frac{2\pi}{n}$.

\begin{figure}[ht]
\begin{center}
 \includegraphics[width=0.35\textwidth]{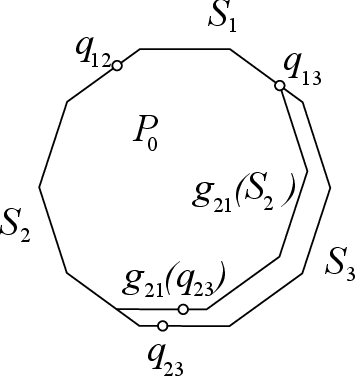}
 \caption{An illustration for Case 1 in Subsection~\ref{subsec:type3}}
\label{fig:type3_case1}
\end{center}
\end{figure}

Finally, in the remaining cases, if $g_{21}(q_{23})$ lies in the interior of $S_1$ with $g_{21}(q_{23}) \in [p_m,p_{m+1}]$ or $g_{21}(q_{23}) \in [p_n,p_1]$, we can apply the argument in the previous case.

\emph{Case 2}, exactly two of $S_1, g_{21}(S_2)$ and $g_{31}(S_3)$ coincide.
Without loss of generality, we may assume that $S_1 = g_{21}(S_2)$.
In the consideration in this case, for brevity, for any $i \neq j$ we let $q_{ij}$ denote the intersection point of $S_i$ and $S_j$, let $q$ denote the unique point in $D_1 \cap D_2 \cap D_3$, and set $C_{ij} = D_i \cap D_j$. By Lemma~\ref{lem:topology}, we have that $C_{ij}$ is a simple (possibly degenerate) curve connecting $q_{ij}$ and $q$.

Note that since $S_3$ cannot contain more than $\frac{n}{2}$ vertices of $P$, we have that each of $S_1$ and $S_2$ contains at least $\frac{n}{4}$ and at most $\frac{n}{2}$ vertices of $P$. This and $n \geq 5$ implies that $g_{21}(P)=P$; that is, $g_{21}$ is an isometry of $P$. In particular, $g_{21}$ (and consequently $g_{12}$) is either a reflection about a symmetry line of $P$, or a rotation around $o$ with angle $\alpha = \frac{2m \pi}{n}$ for some integer $1 \leq m \leq n$. Depending on the type of $g_{21}$, we distinguish two subcases.

\emph{Subcase 2.1}, $g_{21}$ is a rotation around $o$. Then the angle of rotation is $\alpha = \frac{2m \pi}{n}$ for some integer $\frac{n}{4} \leq m < \frac{n}{2}$, which implies, in particular, that $l(S_1)=l(S_2)=m$, and $l(S_3)=n-2m$.  

Observe that since $o$ is a fixed point of $g_{21}$, we have either $o \in D_1 \cap D_2$ or $o \in \inter (D_3)$. First, assume that $o \in D_1 \cap D_2$. If $o=q$, then $g_{21}(q_{12})=q_{13}$, $g_{21}(q_{23})=q_{12}$, and $S_1$ and $S_3$ are congruent, yielding that the tiling is rotationally generated. If $o \neq q$, then $o$ has a closed circular neighborhood $B$ disjoint from $D_3$. Let $t \mapsto C(t)$ be a continuous parametrization of the curve $C_{12}$ with $C(0)=o$, and let $t^+ = \sup \{ t : C([0,t]) \subset B\}$, and 
$t^- = \inf \{ t : C([t,0]) \subset B\}$. Then $g_{21}(C(t^{\pm})) = C(t^{\mp})$, implying that $g_{21}$ is a reflection about $o$, which contradicts the condition that the configuration is Type 3. Thus, we have $o \in \inter (D_3)$.

Let $q_2 = g_{12}(q)$ and $q_1 = g_{21}(q)$. Then $q_2 \in \partial D_2$, $q_1 \in \partial D_1$ and $q_1,q_2 \notin \partial (P)$.
Let $P_0 \subset D_3$ be the homothetic copy of $P$ of maximum homothety ratio centered at $o$. Then $P_0$ touches at least one of the curves $C_{13}$ and $C_{23}$, say, at a point $x_2 \in C_{23}$. Let $x_1 = g_{21}(x_2)$. By the definition of $P_0$, we have $x_1 \in D_3$, and by $x_2 \in D_2$, we have $x_1 \in D_1$. From this it follows that $x_1 \in C_{13}$, implying that $C_{13} \cap g_{21}(C_{23}) \neq \emptyset$. As $C_{12} \cup C_{13}$ is a connected curve from $q_{12}$ to $q_{13}$, and $g_{21}(C_{23})$ is a connected curve in $C_{12} \cup C_{13}$ from $q_{12}$ to $q_1$, the relation $C_{13} \cap g_{21}(C_{23}) \neq \emptyset$ implies that $q_1$ is an interior point of $C_{13}$, from which $q_2 \in int(C_{23})$ also follows. For $i=1,2$, let $C_{i3}^s$ and $C_{i3}^q$ denote the closed arcs of $C_{i3}$ from $q_{i3}$ to $q_i$, and from $q_i$ to $q$, respectively (cf. Figure~\ref{fig:type3_subcase2_1}). Then $g_{21}(C_{23}^s)=C_{12}=g_{12}(C_{13}^s)$ and $g_{21}(C_{23}^q) = C_{13}^q$ yield that the corresponding arcs are congruent.
Thus, by Lemma~\ref{lem:cutcong}, as $\partial D_1 = S_1 \cup C_{13}^q \cup C_{13}^s \cup C_{12}$ and $\partial D_3 = S_3 \cup C_{13}^q \cup C_{13}^s \cup C_{23}^q \cup C_{23}^s$, the equidecomposability of $\partial D_1$ and $\partial D_3$ yields that $S_3 \cup C_{13}^q$ and $S_1$ are equidecomposable, implying that $C_{13}^q$ (and also $C_{23}^q$) is a polygonal curve of length $l(S_1)-l(S_3)$. This yields, in particular, that $l(S_1) > l(S_3)$, $\alpha = \frac{2m \pi}{n} > \frac{2 \pi}{3}$, $m > \frac{n}{3}$, and that $S_1$ contains at least one side of $P$.

\begin{figure}[ht]
\begin{center}
 \includegraphics[width=0.45\textwidth]{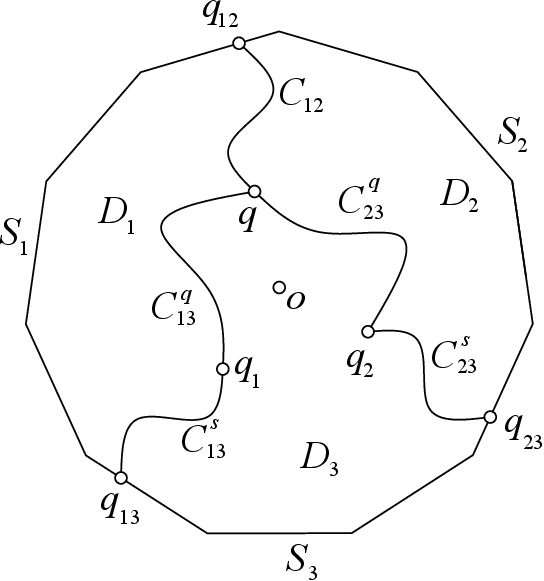}
 \caption{An illustration for Subcase 2.1 in Subsection~\ref{subsec:type3} with $n=11$ and $m=4$.}
\label{fig:type3_subcase2_1}
\end{center}
\end{figure}

Note that as $g_{21}$ is a rotation around $o$, either all the $q_{ij}$ are interior points of some edges of $P$, or all are vertices. Furthermore, $l(S_1)=l(S_2) > l(S_3)$, and $l(C_{13}^q)=l(C_{23}^q)$ are positive integers, and if the $q_{ij}$ are vertices of $P$, then $l(S_3) \geq 2$.
We prove the assertion under the assumption that all the $q_{ij}$ are interior points of some edges of $P$, as in the opposite case a slight modification of our argument can be applied.

Let us call a \emph{copy} of $S_i$ in $\partial D_j$ a subset $S$ of $\partial D_j$ congruent to $S_i$ such that $S$ is not a proper subset of some connected curve $S' \subset \partial D_j$ with the property that the unique congruent copy of $\partial P$ containing $S$ also contains $S'$, and observe that any two copies of $S_i$ are either nonoverlapping or slightly overlapping. Recall that by Lemma~\ref{lem:finite_copies}, for any values of $i$ and $j$, $\partial D_j$ contains finitely many copies of $S_i$.

Consider the case that $q$ is an interior point of a copy $S$ of $S_1$ in $\partial D_3$. Let $S'$ and $S''$ denote the parts of $S$ in $C_{13}^s$ and in $C_{23}^s$, respectively, and assume that $q_1, q_2$ are not interior points of $S$. Then $l(C_{13}^q) = l(S_1)-l(S_3) \leq m-1 = l(S_1)-1$ implies that $S'$ and $g_{21}(S'')$ have a common vertex in their interiors. Thus, they belong to the boundary of the same congruent copy $P'$ of $P$. On the other hand, from this we obtain that $g_{21}$ is a symmetry of $P'$, and hence, $P'$ is centered at $o$, a contradiction. If exactly one of $q_1$ or $q_2$, say $q_2$ is an interior point of $S$, then $S' \subseteq C_{13}^s=g_{21}(S'')$, which yields that $S'$ contains a vertex or it is a segment of length strictly less than one. In the first case we can apply the previous consideration, and in the second case, since by the properties of rotation $S''$ contains a segment of length $l(S')$ that ends at $q_2$, we reach a contradiction with the fact that $l(C_{23}^s) =l(S_1)-l(S_3)$ is a positive integer. Finally, if both $q_1$ and $q_2$ are interior points of $S$, then, by the properties of rotation, $C_{13}^s \cup C_{23}^s \subset S$ yields that the unique regular $n$-gon that contains $S$ in its boundary is $P$; a contradiction.
Thus, we have that $q$ does not belong to the interior of any copy of $S_1$ in $D_3$.

Note that the numbers of copies of $S_1$ in $\partial D_1$ and in $\partial D_3$ are equal. Furthermore, the number of copies of $S_1$ in $\partial D_1$ in $C_{13}^s \cup C_{12}$ is equal to this number in $\partial D_3$ in $C_{13}^s \cup C_{23}^s$. On the other hand, since $l(C_{13}^q)=l(C_{23}^q) < l(S_1)$ and $q$ does not belong to the interior of a copy of $S_1$ in $\partial D_3$, it follows that the number of the copies of $S_1$ in $\partial D_1$ containing an element of $Q_1=\{ q_{12},q_{13},q,q_1\}$ in their interiors is one less than the number of copies in $\partial D_3$ containing an element of $Q_3=\{ q_{13}, q_{23}, q_1,q_2 \}$ in their interiors.

Observe that for $i=1,3$ and any element of $Q_i$, there is at most one copy of $S_1$ in $\partial D_i$ containing it in its interior.
Furthermore, if the interior of every copy of $S_1$ in $\partial D_3$ contains at most one element of $Q_3$, then the number of copies of $S_1$ in $\partial D_3$ whose interiors intersect $Q_3$ is not larger than the number of copies of $S_1$ in $\partial D_1$ whose interiors intersect $Q_1$.
Thus, we may assume that there is a copy $S$ in $\partial D_3$ containing both $q_{13}$ and $q_1$, or both $q_{23}$ and $q_2$. In both cases, we have that $S$ slightly overlaps $S_3$. Since the internal angle of $D_3$ at $q_{13}$ or $q_{23}$, respectively, is obtuse, this yields that both conditions cannot be satisfied simultaneously. Hence, $\partial D_3$ contains exactly one copy of $S_1$, which slightly overlaps $S_3$. Then $g_{13}(S_1)$ coincides with this copy of $S_1$.
Furthermore, $\partial D_i$ is a closed polygonal curve for every value of $i$.

Assume that $q_{13} \in g_{13}(S_1)$, and let $q'$ denote the vertex of $P$ in $S_1$ closest to $q_{12}$. Since $q$ is not an interior point of $g_{13}(S_1)$ and $l(C_{13}^q)=l(S_1)-l(S_3)$, we have that $l(C_{12}) =l(C_{13}^s) \geq l(S_3)-|q'-q_{12}|$. Thus, $q$ is not an interior point of $g_{31}(S_3)$. On the other hand, since $l(S_1)-l(S_3) \geq 1$ and by the definition of a copy, the endpoint of $g_{31}(S_3)$ in $\partial D_1$, different from $q'$, is not an interior point of $C_{12}$. Thus, it follows that this endpoint of $C_{12}$ is $q$. Thus, $l(C_{12}) =l(C_{13}^s) = l(S_3)-|q'-q_{12}|$. This yields that $l(C_{13}^s)+l(C_{13}^q)=l(S_1)-|q'-q_{12}|$, and $q$ is an endpoint of $g_{13}(S_1)$. Now we have completely described the boundaries of the $D_i$, in particular each consists of segments of unit length and some strictly shorter segments. A simple counting shows that $\partial D_1$ contains exactly 4 segments of length strictly smaller than $1$ ($2$ in $S_1$, $1$ in $C_{12}$ and $1$ in $C_{13}^q$), and $\partial D_3$ contains exactly $6$ such segments ($2$ in $S_3$, $1$ in $C_{13}^q$, $2$ in $C_{23}^q$ and $1$ in $C_{23}^s$). This contradicts the fact that $D_1$ and $D_3$ are congruent.

If $q_{23} \in g_{13}(S_1)$, a similar consideration proves the assertion.

\emph{Subcase 2.2}, $g_{21}$ is a reflection about a symmetry line $L$ of $P$. Without loss of generality, we assume that $L$ is the $y$-axis, and the common point of $S_1$ and $S_2$ lies on the positive half of $L$. Note that $g_{12}=g_{21}$, and $D_3 = \cl (P \setminus (D_1 \cup D_2))$ is symmetric to $L$.

Clearly, by Remark~\ref{rem:reflection}, $L$ separates $D_1$ and $D_2$, and from this it readily follows that
$D_1 \cap D_2 = [q_{12},q]$. Furthermore, we have that $D_1$ and $D_2$ are the closures of the subsets of $P \setminus D_3$ contained in the two closed half planes bounded by $L$.

Let $Y= g_{13}(S_1)$. Assume that $Y$ slightly overlaps $S_3$. Then $Y$ crosses $L$. By the symmetry of $D_3$, the reflected copy $Y'$ of $Y$ about $L$ also belongs to $\partial D_3$ and it also crosses $L$. Thus, $Y$ and $Y'$ overlap, and $Y \cup Y'$ intersects $L$ at a right angle. From this we have that $D_3 = \conv D_3$ is the convex region bounded by $S_3 \cup Y \cup Y'$, and an elementary computation yields that $\area(D_3) > \frac{\area(P)}{3}$; a contradiction. Thus, we have that $Y$ does not overlap $S_3$.

Let $Y'$ be the reflected copy of $Y$ about $L$. If $Y$ does not contain $q$ in its interior, then the facts that $l(Y)+l(Y')+l(S_3) = l(P)$ and that $Y, Y', S_3$ are subsets of $\partial (\conv (D_3))$ yield that $D_3 = \conv (D_3) = P$, which contradicts our assumptions. Thus, $Y$ and $Y'$ overlap.  If $Y \cap Y'$ contains a vertex in its interior, then $Y \cup Y'$ belongs to the boundary of a regular $n$-gon, implying that $g_{31}(Y \cup Y') \subset \partial P$. Hence, $Y$ and $Y'$ either slightly overlap, or they coincide. 

Consider the case that $Y$ and $Y'$ slightly overlap, and let $E = Y \cap Y'$. Then $g_{31}(E) \subset g_{31}(Y) = S_1$ lies on the edge containing $q_{12}$, or the edge containing $q_{13}$. Since in the first case $\partial D_1$ and $\partial D_2$ cross, we have that $g_{31}(E)$ lies on the edge containing $q_{13}$; we remark the property that $S_1$ and $g_{31}(Y')$ slightly overlap with $q_{13} \in S_1 \cap g_{31}(Y')$ implies also that $q_{13}$ is not a vertex of $P$.

Let $L_1$ be the supporting line of $P$ parallel to $L$ such that the infinite strip bounded by $L$ and $L_1$ contains $S_1$. If $L_1 \cap P$ is disjoint from $S_1$,
then $D_3$ contains diametrically opposite points of $P$, contradicting our assumptions. Thus, either $q_{13} \in L_1 \cap P$ or $L_1 \cap P$ belongs to the interior of $S_1$. If $q_{13} \in L_1 \cap P$, then $L_1 \cap P$ is a side of $P$, and both endpoints of $g_{31}(Y \cup Y')$ lie on $L$, implying that $D_1 = \conv (D_1) = \conv (g_{31}(Y\cup Y'))$, and thus, $\partial D_1$ does not contain a part congruent to $S_3$; a contradiction. Hence, we are left with the case that $L_1 \cap P$ belongs to the interior of $S_1$. Let $q'=g_{31}(q)$ and let $L'$ be the line intersecting $\partial P$ orthogonally at $q'$ (cf. Figure~\ref{fig:type3_subcase2_2}). Then $q_{12}$ is a unique point in $D_1$ farthest from $q'$. On the other hand, by symmetry, the distances of the two endpoints of $g_{31}(Y \cup Y')$ from $q'$ are the same. Since one of these endpoints is $q_{12}$, it follows that the two endpoints coincide, implying that $g_{31}(Y \cup Y')$ is a simple, closed, convex  curve in $\partial D_1$. Thus, $D_1 = \conv (g_{31}(Y \cup Y'))$, or equivalently, $D_3 = \conv (Y \cup Y')$, which contradicts the assumption that $S_3 \subset \partial D_3$.

\begin{figure}[ht]
\begin{center}
 \includegraphics[width=0.45\textwidth]{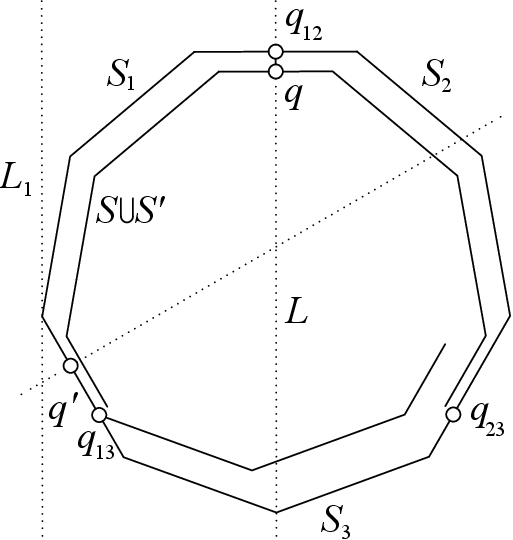}
 \caption{An illustration for Subcase 2.2 in Subsection~\ref{subsec:type3}.}
\label{fig:type3_subcase2_2}
\end{center}
\end{figure}

Finally, we consider the case that $Y=Y'$. Then $Y$ is symmetric to $L$, yielding that $S_1$ is symmetric to the line $L'=g_{31}(L)$. Thus, $L'$ is the bisector of either an edge or an angle of $P$, which implies that $o \in L'$. Since $D_3$ is symmetric to $L$, we also have that $D_1$ is symmetric to $L'$. Let $g$ be the reflection about $L'$, and consider the transformations $g'_{12}, g'_{13}, g'_{23}$ defined by $g'_{12}(\cdot)= g_{12} (g(\cdot))$; $g'_{13} = g_{13}$ and $g'_{23}(\cdot) = g'_{13} (g'^{-1}_{12} (\cdot))$. Clearly, the transformation $g'_{ij}$ is an isometry mapping $D_i$ into $D_j$, and $g'_{12}$ is a rotation around $o$. Thus, in this case we can apply the consideration in Subcase 2.1. 

\emph{Case 3}, all of $S_1, g_{21}(S_2)$ and $g_{31}(S_3)$ coincide.

Since this yields that all of $S_1, g_{21}(S_2)$ and $g_{31}(S_3)$ contain vertices of $P$, it follows that $g_{12}$ and $g_{13}$ are symmetries of $P$. This implies that $o$ is a fixed point of both $g_{21}$ and $g_{31}$, and thus, the unique common point of $D_1$, $D_2$ and $D_3$ is $o$.
If both $g_{21}$ and $g_{31}$ are rotations about $o$, then the tiling is clearly rotationally generated, and we are done.
Assume that, e.g. $g_{21}$ is a reflection about a symmetry line $L$ of $P$. Then Remark~\ref{rem:reflection} implies that $L$ separates $D_1$ and $D_2$, from which we obtain that the curve $D_1 \cap D_2$ is a straight line segment connecting the common point of $S_1$ and $S_2$ to $o$. By the properties of rotations, from this we also have that for any $i \neq j$, $D_i \cap D_j$ is a straight line segment connecting the common point of $S_i$ and $S_j$ to $o$. This, combined with the fact that $l(S_1)=l(S_2)=l(S_3)$, readily yields that the tiling is rotationally generated.

\section{Concluding remarks and open questions}\label{sec:remarks}

\begin{remark}
A simplified version of the proof of Theorem~\ref{thm:main} can be applied to prove the same statement for monohedral tilings of a regular triangle with at most three tiles.
\end{remark}

The authors have found no results about monohedral tilings of convex polygons in spherical or hyperbolic planes.
This is our motivation to state the following problem. Before doing so, we note that the symmetry group of an equiangular convex quadrilateral in spherical or hyperbolic planes contains the symmetry group of a Euclidean rectangle as a subgroup.

\begin{problem}
Let $\mathbb{M}^2$ denote the spherical plane $\mathbb{S}^2$ or the hyperbolic plane $\mathbb{H}^2$, and let $P \subset \mathbb{M}^2$.
Characterize the monohedral tilings of $P$ with at most three discs if $P$ is a
\begin{itemize}
\item[(i)] circular disc;
\item[(ii)] a regular polygon;
\item[(iii)] an equiangular convex quadrilateral.
\end{itemize}
\end{problem}

\section{Data Availability statement}

 Data sharing not applicable to this article as no datasets were generated or analysed during the current study.

\end{document}